\documentclass[ECP]{ejpecp} 

\usepackage{hyperref}
\usepackage{cleveref}

\newcommand{\Reals}{\mathbb{R}}
\newcommand{\Rd}{\mathbb{R}^{d}}
\newcommand{\sset}[2]{\{ #1\mid #2 \}}
\newcommand{\eukl}[2]{\langle #1,#2\rangle}
\newcommand{\supn}[1]{\|#1\|_{\infty}}
\newcommand{\expv}[1]{\mathbb{E}\left[#1\right]}

\newcommand{\eexpv}[2]{\mathbb{E}_{#1}\left[#2\right]}

\newcommand{\Hflow}[2]{\varphi_{#1}(#2)}

\newcommand{\Rdd}{\mathbb{R}^{2d}}
\newcommand{\nd}[2]{\mathcal{N}(#1,#2)}

\newcommand{\lref}{\lambda_{\text{ref}}}

\SHORTTITLE{Cores for PDMPs used in MCMC}

\TITLE{Cores for Piecewise-Deterministic Markov Processes used in Markov Chain Monte Carlo} 

\AUTHORS{%
  Peter Holderrieth\footnote{University of Oxford, United Kingdom
    \EMAIL{peterholderrieth@gmail.com}}}


\AMSSUBJ{60G53; 60J25; 60J76; 60J35} 

\SUBMITTED{July 17, 2021} 
\ACCEPTED{September 24, 2021} 

\VOLUME{26}
\YEAR{2021}
\PAPERNUM{57}
\DOI{10.1214/21-ECP430}

\ABSTRACT{We show fundamental properties of the Markov semigroup of recently proposed MCMC algorithms based on Piecewise-deterministic Markov processes (PDMPs) such as the Bouncy Particle Sampler, the Zig-Zag process or the Randomized Hamiltonian Monte Carlo method.  Under assumptions typically satisfied in MCMC settings, we prove that PDMPs are Feller and that their generator admits the space of infinitely differentiable functions with compact support as a core. As we illustrate via martingale problems and a simplified proof of the invariance of target distributions, these results provide a fundamental tool for the rigorous analysis of these algorithms and corresponding stochastic processes.}

\begin{document}


\section{Introduction}

Markov chain Monte Carlo (MCMC) is a widely-applied class of algorithms which use Markov chains to sample from a probability distribution. For example, they are used for Bayesian inference in machine learning \cite{SMCMIP,article,Freitas2000SequentialMC,Ghahramaniarticle, MCMCRevolution} and play an important role in the simulation of molecular dynamics and statistical physics \cite{APSTPH,MDsimulation,doi:10.1063/1.4929529, Nishikawa_2016}.\\ 
Recently, a new class of MCMC algorithms came up  \cite{Monmarch,PDMCMC,BPSwork} which are based on so-called Piecewise-deterministic Markov Processes (PDMPs) \cite{DAVIS}. Their common idea is to construct a time-continuous Markov process which evolves deterministically up to a certain time  but is interrupted at random times by random jumps. Similiar to classical MCMC algorithms, we can use PDMPs to sample from a target distribution by designing the deterministic dynamics and the jump mechanism in a way such that the target distribution is an invariant distribution of the PDMP. Among those recently proposed piecewise-deterministic MCMC schemes are the Bouncy Particle Sampler (BPS) \cite{BPSwork}, the Zig-Zag process \cite{bierkens2016zigzag}, the Randomized Hamiltonian Monte Carlo method (RHMC)\cite{RHMCintro}, and their variations \cite{wu2017generalized,PDMCMC}.\\ 
Despite their empirical success 	\cite{bierkens2016zigzag,PDMCMC}, these algorithms and their corresponding PDMPs are not well understood yet. As there is no closed form of the Markov semi-group $P_t$, analysis of these algorithms is mainly based on their generator $\mathcal{L}$ \cite{DAVIS}. However, the functions $f:\Rd\to\Reals$ of the domain $\text{dom}(\mathcal{L})$ of $\mathcal{L}$ often lack properties such as differentiability, boundedness or compact support. This makes mathematical analysis hard and technical. Even to give a rigorous proof of the invariance of the target distribution under the Markov process  - a fundamental property for their applicability -  can often be laborious \cite{RHMCintro,Monmarch}.\\
The goal of this work is to help out here: we give sufficient conditions for the PDMP to be a Feller process and for the generator $\mathcal{L}$ to admit the space $C_c^\infty(\Rd)$  of compactly supported, infinitely differentiable functions as a core. This simplifies analysis significantly. For example, it follows that infinitesimal invariance of a probability measure on $C_c^\infty(\Rd)$ immediately implies invariance and we can equivalently characterize a PDMP process via martingale problems on $C_c^\infty(\Rd)$. With this work, we therefore hope to provide a useful tool for other researchers to make further progress in the rigorous analysis of these algorithms and their stochastic processes.\\
This work is structured as follows. In \cref{sec:definition}, we define PDMPs on $\Rd$, the basic concept of this work. In \cref{sec:Jabobi}, we prove a new inequality for the "Jacobian" of PDMPs, which provides the basis for the following sections. In \cref{sec:Feller}, we prove that PDMPs are Feller under assumptions typically satisfied in MCMC settings. In \cref{sec:main_result}, we combine these results to the main theorem of this work identifying $C_c^\infty(\Rd)$ as a core of PDMPs under reasonable assumptions. In \cref{sec:examples}, we show how our result can be applied to popular MCMC schemes. Finally, we show in \cref{sec:applications} how our results can be applied in the analysis of MCMC algorithms.

\section{Piecewise-deterministic Markov processes}
\label{sec:definition}
We start by giving a definition of a Piecewise-deterministic Markov process (PDMP) on $\Rd$ \cite{DAVIS}. The deterministic dynamics of a PDMP are given by the ordinary differential equation (ODE)
\begin{align}
	\label{ODE}
	\frac{d}{dt} x(t)=g(x(t)), \quad x(0)=z\in\Rd,
\end{align}
where $g\in C^{1}(\Rd,\Rd)$ is a vector field. We impose the following assumption on $g$:
\begin{assumption}[Lipschitz continuity]
	\label{Lipschitz}
	The function $g\in C^{1}(\Rd,\Rd)$ is Lipschitz continuous, i.e. there exists a $L>0$ such 
	that
	\begin{align*}
		&\|g(x)-g(y)\|\leq L \|x-y\| \quad \text{for all } x,y\in\Rd,\\
		\Leftrightarrow &\|Dg(z)\|\leq L \quad \text{for all } z\in\Rd,
	\end{align*}
where $\|\cdot\|$ denotes the operator norm with respect to the Euclidean norm.
\end{assumption}

By assumption \ref{Lipschitz}, the ODE 
is described by a Lipschitz continuous $C^{1}$-vector field. By standard results about ODEs \cite[theorem 2.5.6, theorem 7.3.1]{ODEs}, uniqueness of the solution holds and there is a function $\varphi\in C^{1}(\mathbb{R}^{d+1},\Rd)$, $(t,z)\mapsto \Hflow{t}{z}$, such that
$t\mapsto\Hflow{t}{z}$ solves \cref{ODE} for all $z\in\Rd$. Then $(\varphi_{t})_{t\in\Reals}$ is a group of $C^{1}$-diffeomorphisms on $\Rd$.\\
In a PDMP, the deterministic dynamics are interrupted by random jumps at random times. To describe the random jumps,  we are given a probability space $(S, \mathcal{S},\Xi)$ and a function $R:\Rd\times S\to\Rd,(z,\xi)\mapsto R^\xi(z)$. For a fixed $\xi\in S$, $R^\xi:\Rd\to\Rd$ describes a "jump map" and $\xi\sim\Xi$ describes the random sample of such a jump map.\\
The random jump times are given by an inhomogeneous Poisson process with a continuous intensity function $\lambda:\Rd\to\Reals_{\geq 0}$. More specifically, given the process is at $z\in\Rd$ at time $t_0=0$, the distribution $\mu_{z}$ of the time until the next jump is determined by
\begin{align*}
	\mu_{z}([t,\infty))=\exp{(-\int\limits_{0}^{t}\lambda(\varphi_{s}(z))ds)} \text{ for all } t\geq 0.
\end{align*}
We rewrite $\mu_z$ by a function $W:\Rd\times (0,1)\to\Reals_{\geq 0}$ such that $W(z,U)\sim\mu_{z}$ if $U\sim \text{Unif}_{(0,1)}$.
\begin{definition}[PDMP]
	Let $U_{1},U_2,\dots\sim\text{Unif}_{(0,1)}$ and $\xi_{1},\xi_{2},...\sim\Xi$ be independent random variables. Fix a starting position $z\in\Rd$. 
	Define jumping times $T_k$ and positions $Z_{T_k}$ at jumping times recursively: set $T_{0}=0$, $Z_{0}=z$ and 
	\begin{align*}
		T_{k}=W(Z_{T_{k-1}},U_{k})+T_{k-1},\quad
		Z_{T_{k}}=R^{\xi_{k}}(\varphi_{T_{k}-T_{k-1}}(Z_{T_{k-1}})).
	\end{align*}
	Finally, define the position between jumps $Z_t$. Let $N_{t}=\text{max}\sset{k\in\mathbb{N}}{T_{0}+...+T_{k}\leq t}$ and 
	\begin{align*}
		Z_{t}:=\varphi_{t-T_{N_{t}}}\circ R^{\xi_{k}}\circ\varphi_{T_{N_{t}}-T_{N_{t}-1}}\circ R^{\xi_{k-1}}
		...\circ R^{\xi_{1}}\circ\varphi_{T_{1}}(z)\quad \text{ for all } t>0.
	\end{align*}
\end{definition}
A process $Z_t$ defined like that is called a \emph{Piecewise-deterministic Markov process (PDMP)}.

\section{Grönwall-Jacobi inequality}
\label{sec:Jabobi}
In this section, we prove a Grönwall-like inequality for PDMPs, which will provide the basis for later sections. For all $t\in\Reals$, we write $D\varphi_t$ to denote the Jacobian of the function $\varphi_t:\Rd\to\Rd$. 
\begin{lemma}[Grönwall-Jacobi for deterministic dynamics]
	\label{Jacobiannorm}
	Suppose that assumption \ref{Lipschitz} holds. Then for all $z\in\Rd$ and $t\geq 0$: $\|D\varphi_{t}(z)\|\leq\exp{(Lt)}$ where $\|\cdot\|$ denotes the operator norm with respect to the Euclidean norm.
\end{lemma}
\begin{proof}
	By assumption \ref{Lipschitz}, there exists a $L>0$ such that 
	$\|Dg(z)\|\leq L$ for all $z\in\Rd$. Let $z,x\in\Rd$ be arbitrary and define $f(t):=\|D\varphi_{t}(z)x\|^{2}$. Then 
	\begin{align*}
	\frac{d}{dt}f(t)
	=2\eukl{\frac{d}{dt}D\varphi_{t}(z)x}{D\varphi_{t}(z)x}
	&=2\eukl{D\frac{d}{dt}\varphi_{t}(z)x}{D\varphi_{t}(z)x}\\
	&=2\eukl{D(g\circ\varphi_{t})(z)x}{D\varphi_{t}(z)x}\\
	&=2\eukl{Dg(\varphi_{t}(z))D\varphi_{t}(z)x}{D\varphi_{t}(z)x}\\
	&\leq 2L
	\|D\varphi_{t}(z)x\|^{2}\\&=2L f(t).
	\end{align*}
	By Grönwall's inequality \cite{Gronwall}, it follows that:
	\begin{align*}
	\|D\varphi_{t}(z)x\|^{2}=f(t)\leq f(0)\exp{(2Lt)}=\|x\|^{2}\exp{(2Lt)}
	\end{align*}
	and hence $\|D\varphi_{t}(z)\|\leq \exp{(Lt)}$.
\end{proof}
Consider the example on $\Reals$ with $g(x)=x$ and $L=1$. Then $\varphi_{t}(z)=z\exp{(t)}$ and $|D\varphi_{t}(z)|=|\exp{(t)}|$. So one can see that the bound is actually sharp in this case.
\begin{assumption}
	\label{Jacobirefreshmentbound}
	For every $\xi\in S$ it holds that $R^{\xi}:\Rd\to\Rd$ is in $C^{1}(\Rd,\Rd)$ and subcontractive:
	\begin{align*}
		& \|R^\xi(x)-R^\xi(y)\|\leq \|x-y\| \quad \text{for all  } x,y\in\Reals^d,\\
		\Leftrightarrow &\|DR^{\xi}(z)\|\leq 1 \quad \text{for all  } z\in\Rd.
	\end{align*}
\end{assumption}
Assumption \ref{Jacobirefreshmentbound} states that the jump maps do not enlarge distances between points. Intuitively, this allows to control the position of a PDMP after a jump. Using this assumption, 	Lemma \ref{Jacobiannorm} can be further extended to a process allowing for jumps:
\begin{proposition}[Grönwall-Jacobi for PDMPs]
	\label{boundJacobi2}
	Suppose that assumption \ref{Lipschitz} and \ref{Jacobirefreshmentbound} hold.
	Then for all $t\geq 0$, $k\in\mathbb{N}$, $t_{0},t_{1},...t_{k}\geq 0$ such that $t_{0}+...+t_{k}=t$, $\xi_{1},...\xi_{k}\in S$ and $z\in\Rd$:
	\begin{equation*}
	\|D(\varphi_{t_{k}}\circ R^{\xi_{k}}\circ\varphi_{t_{k-1}}\circ R^{\xi_{k-1}}\circ
	...\circ R^{\xi_{1}}\circ\varphi_{t_{0}})(z)\|\leq\exp{(Lt)}.
	\end{equation*}
\end{proposition}
\begin{proof}
	For $k=0$ this is Lemma \ref{Jacobiannorm}. Assume that the statement is true for $k-1$. Define $h:=\varphi_{t_{k-1}}\circ R^{\xi_{k-1}}\circ
	...\circ R^{\xi_{1}}\circ\varphi_{t_{0}}$. By assumption, it holds that
	$\|Dh\|\leq\exp{(L(t_{0}+...+t_{k-1}))}$. Hence
	\begin{align*}
	\|D(\varphi_{t_{k}}\circ R^{\xi_{k}}\circ h)(z)\|\leq&
	\|D\varphi_{t_{k}}(R^{\xi_{k}}(h(z)))\|
	\|DR^{\xi_{k}}(h(z))\|
	\|Dh(z))\|\\
	\leq&\exp{(Lt_{k})}\exp{(L(t_{0}+...+t_{k-1}))}\\
	=&\exp{(Lt)},
	\end{align*}
	where in the last inequality we used Assumption \ref{Jacobirefreshmentbound}.
\end{proof}
\section{PDMPs and the Feller property}
\label{sec:Feller}
Let $B(\Rd)$ be the space of bounded, measurable functions $f:\Rd\to\Reals$, let $C_0(\Rd)$ be the subspace of all continuous function that vanish at infinity and let $C_c^k(\Rd)$ be the subspace of $k$-times continuously differentiable functions with compact support. \\
Let $P_{t}$ be the Markov semigroup of a PDMP $Z_{t}$, i.e. $P_{t}f(z)=\eexpv{z}{f(Z_{t})}$ for all $f\in B(\Rd)$. The main goal of this section is to show fundamental properties of $P_t$, in particular that it is Feller. Define the space on which $P_t$ is strongly-continuous
\begin{align*}
		B_{P}(\Rd)&:=\sset{f\in B(\Rd)}{\lim\limits_{t\to 0}\supn{P_{t}f-f}=0}.
\end{align*}
The extended generator $\mathcal{A}$ is defined for all functions $f:\Rd\to\Reals$ such that there is a function $\mathcal{A}f:\Reals^d\to\Reals$ for which $t\mapsto \mathcal{A}f(Z_{t})$ is almost surely integrable on bounded intervals and
\begin{align}
	\label{eq:local_martingale}
	M_t^f=f(Z_t)-f(Z_0)-\int\limits_{0}^{t}\mathcal{A}f(Z_s)ds
\end{align}
is a local martingale \cite{DAVIS}. For the scope of this work, it is sufficient that every differentiable and bounded $f$ is in the domain of $\mathcal{A}$ and it holds
\begin{align}
	\label{eq:generator}
	\mathcal{A}f=\eukl{\nabla f(z)}{g(z)}+\lambda(z)(Qf(z)-f(z))
\end{align}
with $Qf(z)=\expv{f(R^\xi(z))}$ for $\xi\sim\Xi$ (see \cite[theorem 26.14]{DAVIS}).
\begin{assumption}
 	\label{Feller_assumptions}
 	One of the following two conditions is true:
 	\begin{enumerate} 
 		\item The intensity function is bounded, i.e. $\supn{\lambda}<\infty$, and as $\|z\|\to\infty$
 		\begin{align*}
 			\inf\limits_{0\leq s\leq t}\|\varphi_s(z)\|\to\infty \text{ for all } t>0,\quad \|R^{\xi}(z)\|\to\infty  \text{ for all } \xi\in S.
 		\end{align*}
 		\item Jumps are isometric, i.e. $\|R^\xi(z)\|=\|z\|$ for all $z\in\Rd$ and $\xi\in S$, and as $\|z\|\to\infty$
 		\begin{align*}
 			\inf\limits_{\substack{t_0+\dots+t_k=t\\t_i\geq 0\\
 			\xi_1,\dots,\xi_{k}\in \mathcal{S}}}\|\left [\varphi_{t_k}\circ R^{\xi_k}\circ\dots\circ R^{\xi_{1}}\circ\varphi_{t_0}\right ](z)\|\to \infty.
 		\end{align*}
 \end{enumerate}
\end{assumption}

\begin{proposition}[Feller]
	\label{Feller}
	 Let Assumption \ref{Lipschitz} and \ref{Feller_assumptions} be true, then $P_{t}$ is Feller, i.e. its semigroup $P_{t}$ satisfies the following two conditions:
	\begin{enumerate} 
		\item \textbf{Feller property}: $\forall t\geq 0, \forall f \in C_{0}(\Rd): P_{t}f\in C_{0}(\Rd) $ 
		\item \textbf{Strong continuity}: $C_0(\Rd)\subset B_P(\Rd)$
	\end{enumerate}
\end{proposition}

\begin{proof} We begin by proving strong continuity.
\paragraph{Strong continuity.} 
Let $f\in C_{c}^1(\Rd)$. If $R^s$ is isometric, $\mathcal{A}f$ is continuous and has compact support. If $\supn{\lambda}<\infty$ both summands in \cref{eq:generator} are uniformly bounded. In both cases, the local martingale in 	\cref{eq:local_martingale} is a true martingale due to the uniform bound. By taking expectation values in \cref{eq:local_martingale} and using Fubini, one has for all $z\in \Rd $:
 	\begin{align}
 		\label{Doeblin}
 		&P_{t}f(z)-f(z)-\int_{0}^{t}P_{s}\mathcal{A}f(z)ds=0\\
 		\Rightarrow&|P_{t}f(z)-f(z)|\leq\int_{0}^{t}\supn{P_{s}\mathcal{A}f}ds \leq t\supn{\mathcal{A}f} \to 0 
 	\end{align}
 	uniformly in $z$ as $t\to 0$. Consequently, $C_{c}^{1}(\Rdd)\subset B_{P}(\Rd)$. It is well-known that $B_{P}(\Rd)$ is a closed subspace and $C_{c}^{1}(\Rdd)\subset C_{0}(\Rdd)$ is dense \cite{Stone-Weierstrass}. By this, it follows that also $C_{0}(\Rdd)\subset B_{P}(\Rd)$.

\paragraph{Feller property.} Assume condition (1) of Assumption \ref{Feller_assumptions} is true. Let $f\in C_0(\Rd)$ and $\epsilon>0$. Let $N_t$ be the number of jumps up to time $t$. Since $\lambda$ is bounded, we a find a $k\in\mathbb{N}$ such that the probability of cases where $N_t>k$ is less than $\epsilon/\supn{f}$. But then
\begin{align*}
|P_tf(z)|\leq\mathbb{E}_z[|f(Z_t)|;N_t\leq k]+\epsilon.
\end{align*}
For all cases where $N_t\leq k$, we find that the norm of 
\begin{align*}
		Z_{t}:=\varphi_{t-T_{N_{t}}}\circ R^{\xi_{N_t}}\circ\varphi_{T_{N_{t}}-T_{N_{t}-1}}\circ R^{\xi_{N_t-1}}
		...\circ R^{\xi_{1}}\circ\varphi_{T_{1}}(z)
\end{align*}
goes to infinity as $\|z\|\to\infty$. (the jumping times and $N_t$ might change but this does not matter since the dynamics go to infinity uniformly over $0\leq s\leq t$ and $N_t$ is bounded). Therefore, $f(Z_t)\to 0$ as $\|z\|\to\infty$ for all cases when $N_t\leq k$. By dominated convergence, we find that 
\begin{align*}
\limsup\limits_{\|z\|\to\infty} |P_tf(z)|\leq \epsilon.
\end{align*}
Since $\epsilon>0$ was arbitary, we can conclude that $P_tf\in C_0(\Rd)$.\\
If condition (2) of Assumption \ref{Feller_assumptions} is true, it is clear that $\|Z_t\|\to\infty$ as $\|z\|\to\infty$  and therefore $f(Z_t)\to 0$ almost surely. By dominated convergence, we find that $P_tf(z)\to 0$ as $\|z\|\to\infty$ and hence $P_tf\in C_0(\Rd)$.
\end{proof}
Consider now $P_{t}$ as a semigroup on $C_{0}(\Rd)$. Let $\mathcal{L}$ be its (strong) generator on $C_{0}(\Rd)$ defined by
\begin{align*}
\mathcal{L}f=\lim\limits_{t\to 0}\frac{1}{t}(P_{t}f-f) \quad \text{ for } f \in \text{dom}(\mathcal{L}):=\sset{f\in C_{0}(\Rd)}{\lim\limits_{t\downarrow 0}\frac{1}{t}(P_{t}f-f) \text{ exists }}.
\end{align*}
As the proof of the next lemma shows, the extended generator $\mathcal{A}$ and the strong generator $\mathcal{L}$ coincide on $\text{dom}(\mathcal{L})$ under suitable assumptions. These assumptions often bound the non-local part of $\mathcal{A}$ (second summand in \cref{eq:generator}) such that jumps can be neglected for $t\to 0$.
\begin{lemma}
 		\label{Stronggenerator}
 		Let Assumption \ref{Lipschitz} and \labelcref{Feller_assumptions} be true. Then it holds that $C_{c}^{1}(\Rd)\subset \text{dom}(\mathcal{L})$.
\end{lemma}
 	\begin{proof}
 	Assume for now that $\mathcal{A}f\in B_{P}(\Rd)$ for all $f\in C_{c}^{1}(\Rdd)$. Then by \cref{Doeblin}
 	\begin{align*}
 		\supn{\frac{P_{t}f-f}{t}-\mathcal{A}f}&
 		\leq\frac{\int_{0}^{t}\supn{P_{s}\mathcal{A}f-\mathcal{A}f}ds}{t} \to 0
 		\quad\text{as } t\to 0
 	\end{align*}		
 	since $\mathcal{A}f\in B_{P}(\Rd)$, which implies that $\mathcal{L}f=\mathcal{A}f$ and $f\in \text{dom}(\mathcal{L})$.
 	So it remains to show that $\mathcal{A}f\in B_{P}(\Rd)$.\\ If jumps are isometric, then it holds that $\mathcal{A}f$ has the same support as $f$, in particular compact support, and therefore is in $C_0(\Rd)\subset B_P(\Rd)$. If condition (2) of Assumption \ref{Feller_assumptions} holds, we can show directly by dominated convergence that $\mathcal{A}f(z)\to 0$ as $\|z\|\to\infty$, i.e. $\mathcal{A}f\in C_0(\Rd)\subset B_P$ as shown above. This finishes the proof.
 \end{proof}

\section{Cores for PDMPs}
\label{sec:main_result}
Knowing that $P_t$ is Feller is advantageous since $C_{0}(\Rd)$ consists of "nice" functions and due to strong continuity we can restrict our attention to regular dense subsets such as $C_c^1(\Rd)$. However, Markov processes are often studied via their strong generator $\mathcal{L}$. For example, in many cases we have no analytical expression for $P_t$, while the generator of a PDMP is explicitly given (see \cref{eq:generator}, note that $\mathcal{A}=\mathcal{L}$ on $\text{dom}(\mathcal{L})$).\\ 
Naturally, the question arises whether such "sufficient, regular subsets" exists for $\text{dom} (\mathcal{L})$ as they do for $P_t$. In contrast to the operators $P_{t}$, $\mathcal{L}$ is not continuous in general and therefore a mere dense subset is not "sufficient". That is why one searches for cores, a fundamental concept in the study of semigroups \cite{Ethier,Pazy}. A core of $\mathcal{L}$ is a subset $D\subset \text{dom} (\mathcal{L})$ such that for all $f\in \text{dom} (\mathcal{L})$ there exists a sequence $f_{n}\in D$ such that 
\begin{equation*}
\lim\limits_{n\to\infty}\supn{f_{n}-f}=0,
\quad 
\lim\limits_{n\to\infty}\supn{\mathcal{L}f_{n}-\mathcal{L}f}=0.
\end{equation*}
\begin{lemma}
	\label{crucialstepcore}
	Suppose Assumption \ref{Lipschitz}, \labelcref{Jacobirefreshmentbound} and	\labelcref{Feller_assumptions} hold.
	Define $D:=\sset{h\in \text{dom} (\mathcal{L})}
	{h\in C^{1}(\Rd), \partial _{i}h\in C_{0}(\Rd) \,\, i=1,...,d}$.
	Then for all $f\in C_{c}^{1}(\Rd)$ it holds that  $P_{t}f\in D$ for all $t\geq 0$.
\end{lemma}
\begin{proof}
	Let $f\in C_{c}^{1}(\Rd)$.  $P_{t}f\in C_{0}(\Rd)$ holds by Assumption \ref{Feller}. Since $P_{t}$ is a contraction semigroup and $f\in \text{dom} (\mathcal{L})$ by \cref{Stronggenerator}, one has $P_{t}f\in \text{dom} (\mathcal{L})$ and $\mathcal{L}f\in C_{0}(\Rd)$. To show differentiability, rewrite
	\begin{align}
	\label{question1}
	P_{t}f(z)
	= \eexpv{z}{f(Z_{t})}
	=\expv{f\circ\varphi_{t-T_{N_{t}}}\circ R^{\xi_{k}}\circ\varphi_{T_{N_{t}}-T_{N_{t}-1}}\circ R^{\xi_{k-1}}
		...\circ R^{\xi_{1}}\circ\varphi_{T_{1}}(z)}.
	\end{align}
	 The function within the expectation is differentiable in $z$. By dominated convergence, it therefore suffices to show that the gradient is uniformly bounded:
	\begin{align*}
	&\|\nabla (f\circ\varphi_{t-T_{N_{t}}}\circ R^{\xi_{k}}\circ\varphi_{T_{N_{t}}-T_{N_{t}-1}}\circ R^{\xi_{k-1}}
	...\circ R^{\xi_{1}}\circ\varphi_{T_{1}})(z))\|\\
	=&\|(\nabla f(\varphi_{t-T_{N_{t}}}\circ R^{\xi_{k}}\circ\varphi_{T_{N_{t}}-T_{N_{t}-1}}\circ R^{\xi_{k-1}}
	...\circ R^{\xi_{1}}\circ\varphi_{T_{1}}(z)))^{T}\\
		\nonumber
	&\,D(\varphi_{t-T_{N_{t}}}\circ R^{\xi_{k}}\circ\varphi_{T_{N_{t}}-T_{N_{t}-1}}\circ R^{\xi_{k-1}}
	...\circ R^{\xi_{1}}\circ\varphi_{T_{1}})(z)\|\\
	\leq&\supn{\nabla f}\exp{(Lt)}
	\end{align*}
	by the Grönwall-Jacobi bound in Assumption \ref{boundJacobi2}. One can conclude that:
	\begin{align*}
	\partial_{j}P_{t}f(z)
	=\mathbb{E}[(&\nabla f(\varphi_{t-T_{N_{t}}}\circ R^{\xi_{k}}\circ\varphi_{T_{N_{t}}-T_{N_{t}-1}}\circ R^{\xi_{k-1}}
	...\circ R^{\xi_{1}}\circ\varphi_{T_{1}}(z)))^{T}\\
	\nonumber
	&\,\partial_{j}(\varphi_{t-T_{N_{t}}}\circ R^{\xi_{k}}\circ\varphi_{T_{N_{t}}-T_{N_{t}-1}}\circ R^{\xi_{k-1}}
	...\circ R^{\xi_{1}}\circ\varphi_{T_{1}})(z)],
	\end{align*}
	and by dominated convergence one can conclude that $\partial_{j}P_{t}f$ is continuous and bounded by $\supn{\nabla f}\exp{(Lt)}$. 
	In addition,
	\begin{align}
	|\partial_{j}P_{t}f(z)|
	\leq&\exp{(Lt)}\expv{\|\nabla f(\varphi_{t-T_{N_{t}}}\circ R^{\xi_{k}}\circ\varphi_{T_{N_{t}}-T_{N_{t}-1}}\circ R^{\xi_{k-1}}
		...\circ R^{\xi_{1}}\circ\varphi_{T_{1}}(z))\|}\\
	=&\exp{(Lt)}\eexpv{z}{\|\nabla f(Z_{t})\|}.
	\label{eq:bound_grad_f}
	\end{align}
	Since $\|\nabla f\|\in C_{c}(\Rd)$, by the Feller property $P_{t}(\|\nabla f\|)=\eexpv{z}{\|\nabla f(Z_{t})\|}\in C_{0}(\Rd)$. The bound in \cref{eq:bound_grad_f} then implies that $\partial_{j}P_{t}f$ vanishes at infinity as well. Combined with the continuity, we can conclude that  $\partial_{j}P_{t}f\in C_0(\Rd)$.
\end{proof}
\begin{corollary}
	Suppose that Assumption \ref{Lipschitz}, \labelcref{Jacobirefreshmentbound} and \labelcref{Feller_assumptions} hold.
	Then the subspace $D$ is a core of the generator $\mathcal{L}$ of the semigroup $P_{t}$ considered as a semigroup on $C_{0}(\Rd)$.
\end{corollary}
\begin{proof}
	Since $C_{c}^{1}(\Rd)\subset C_{0}(\Rd)$ is dense \cite{Stone-Weierstrass}, $C_{c}^{1}(\Rd)\subset D$ and $P_{t}:C_{c}^{1}(\Rd)\to D\subset \text{dom} (\mathcal{L})$ by Lemma \ref{crucialstepcore}, it follows by \cite[proposition 1.3.3]{Ethier} that $D$ is a core of $\mathcal{L}$.
\end{proof}
We can further reduce $D$ to a more regular subset:
\begin{lemma}
	\label{C1ccore}
	Suppose that Assumption \ref{Lipschitz}, \labelcref{Jacobirefreshmentbound}, and 	\labelcref{Feller_assumptions} hold.
	Then the set $C_{c}^{1}(\Rd)$ is a core of $\mathcal{L}$.
\end{lemma}
\begin{proof}
	Since $C_{c}^{1}(\Rd)\subset D$ and $D$ is a core, it suffices to show that for all $f\in D$ there exist $f_{n}\in C_{c}^{1}(\Rd)$ such that 
	$\supn{f_{n}-f}\to 0$ and $\supn{\mathcal{L}f_{n}-\mathcal{L}f}\to 0$ as $n\to\infty$.\\
	Let $B(0,R):=\{x\in\Rd | \|x\|< R\}$.
	Let $f\in D$ be arbitrary and let $\eta$ be  a $C_{c}^{\infty}(\Rd)$-function such that
	$\eta\equiv 1 $ on $B(0,1)$, $\eta\equiv 0 $ on $\Rd\setminus B(0,c)$ for some $c>1$, $0\leq\eta\leq 1$ and $\supn{\nabla\eta}\leq 1$. One can choose $\eta$ by $\eta(x)=\phi(\|x\|)$ for some function $\phi:\Reals\to\Reals$, in particular it holds then  $\eta(x)=\eta(y)$ whenever $\|x\|=\|y\|$.
	Define $\eta_{k}(x):=\eta(\frac{1}{k}x)$. Then
	\begin{align*}
	\eta_{k} \equiv 1 \text{  on } B(0,k), \quad \eta_{k}\equiv 0 \text{  on  } \Rd\setminus B(0,kc), \quad \supn{\nabla\eta_{k}}\leq\frac{1}{k}.
	\end{align*} 
	Then for $f_{k}:=\eta_{k}f\in C_{c}^{1}(\Rd)$, one has 
	\begin{align}
	\label{eq:whatever}
	\supn{f_{k}-f}\leq\sup\limits_{x\notin B(0,k)}|f(x)|\to 0 
	\end{align}
	as $k\to\infty$ since $f\in C_{0}(\Rd)$.
	Moreover, using $\nabla(\eta_{k}f)=f\nabla\eta_{k}+\eta_{k}\nabla f$:
	\begin{align*}
	\mathcal{L}f_{k}=\eukl{g}{f\nabla\eta_{k}+\eta_{k}\nabla f}+
	\lambda(Qf_{k}-f_{k})
	=&f\eukl{g}{\nabla\eta_{k}}+\eta_{k}\eukl{g}{\nabla f}+
	\lambda(Qf_{k}-f_{k}),
	\end{align*}
	and therefore
	\begin{align}
	&\|\mathcal{L}f-\mathcal{L}f_{k}\|_{\infty}\\
	\label{eq:conv_of_generators_term}
	\leq&\supn{f\eukl{g}{\nabla\eta_{k}}}+
	\supn{
		(1-\eta_{k})(\eukl{g}{\nabla f})+\lambda(
		\eexpv{z}{\left[(1-\eta_{k})f\right](R^\xi(z))}-(1-\eta_{k})f)
	}
	\end{align}
	Define $B_{k}:=\overline{B(0,kc)}\setminus B(0,k)$ and compute using Assumption \ref{Lipschitz}
	\begin{align*}
	\supn{
		f\eukl{g}{\nabla\eta_{k}}}
	\leq\sup\limits_{z\in B_{k}}\Bigl(|f(z)|\|g(z)\|\|\nabla \eta_{k}(z)\|\Bigr)
	\leq&\sup\limits_{z\in B_{k}}\Bigl(|f(z)|(L\|z\|+\|g(0)\|) \frac{1}{k}\Bigr)\\
	\leq&(Lc+\|g(0)\|)\sup\limits_{z\in B_{k}}|f(z)|
	\end{align*}
    which goes to zero as $k\to\infty$ since $f\in C_0(\Rd)$.\\
	Next, we show that the second term in 	\cref{eq:conv_of_generators_term} goes to zero if one of the conditions in Assumption \ref{Feller_assumptions} holds.
	Firstly, assume condition (2). Then, since $\|R^\xi(z)\|=\|z\|$, also  $\eta_k(R^\xi(z))=\eta_k(z)$ and the second term becomes
	\begin{align*}
	\supn{
		(1-\eta_{k})\mathcal{L}f
	}\leq\sup\limits_{z\notin B_{k}}|\mathcal{L}f(z)|,
	\end{align*}
	which converges to $0$ as $k\to\infty$ since $\mathcal{L}f\in C_{0}(\Rd)$.\\
	Secondly, assume condition (1). Then one can show by dominated convergence that  $Qf\in C_{0}(\Rd)$ and $\eukl{g}{\nabla f}=\mathcal{L}f-\lambda(Qf-f)$ is in $C_{0}(\Rd)$.
	So we can bound the second term by
	\begin{align*}
	\supn{
		(1-\eta_{k})(\eukl{g}{\nabla f})}+2\supn{\lambda}\supn{f-f_{k}}.
	\end{align*}
	This goes to zero since $\eukl{g}{\nabla f}\in C_0(\Reals^d)$ and due to 	\cref{eq:whatever}.
\end{proof}
\begin{theorem}
	\label{Cinftycore}
		Suppose that Assumption \ref{Lipschitz}, \labelcref{Jacobirefreshmentbound}, and 	\labelcref{Feller_assumptions} hold. Then
	$C_{c}^{\infty}(\Rd)$ is a core of the generator $\mathcal{L}$ of the semigroup $P_{t}$ considered as a semigroup on $C_{0}(\Rd)$.
\end{theorem}
\begin{proof}
	Let $f\in C_{c}^{1}(\Rd)$ be arbitrary. Choose $f_{n}\in C_{c}^{\infty}(\Rd)$ such that $\supn{f_{n}-f}$ and $\supn{\nabla f_{n}-\nabla f}$ go to zero as $n\to\infty$ and such that $\text{supp}f,\text{supp}f_{n}\subset B(0,R)$ for a $R>0$ (for example choose $f_{n}:=f*\eta_{k}$ for a mollifier $\eta_{k}$). With a similiar computation as above, one sees that $\mathcal{L}f_{n}\to \mathcal{L}f$ uniformly:
	\begin{align*}
	&\|\mathcal{L}f_{n}-\mathcal{L}f\|_{\infty}\\
	=&\|\lambda(Q(f_{n}-f)-(f_{n}-f))+\eukl{\nabla f_{n}-\nabla f}{g}\|_{\infty}\\
	\leq&\|\lambda Q(f_{n}-f)\|_{\infty}+ \sup\limits_{z\in B(0,R)}|\lambda(z)|\supn{f_{n}-f}
	+\sup\limits_{z\in B(0,R)}\Bigl(|\eukl{\nabla f_{n}(z)-\nabla f(z)}{g(z)}|\Bigr)\\
	\leq&\|\lambda Q(f_{n}-f)\|_{\infty}+ \sup\limits_{z\in B(0,R)}|\lambda(z)|\supn{f_{n}-f}
	+(LR+\|g(0)\|)\|\nabla f_{n}-\nabla f\|_{\infty}.
	\end{align*}
	Clearly, the second and the third term converge to zero. What about the first term? If jumps are isometric $\|R^\xi(z)\|=\|z\|$, then
	\begin{align*}
	\|\lambda Q(f_{n}-f)\|_{\infty}\leq \sup\limits_{z\in B(0,R)}|\lambda(z)|\supn{f_{n}-f}\to 0 \text{ as } n\to\infty.
	\end{align*}
 If $\lambda$ is bounded (condition 1 of Assumption \ref{Feller_assumptions}), one similarly gets that 
 \begin{align*}
 \|\lambda Q(f_{n}-f)\|_{\infty}\leq\supn{\lambda}\|Q(f_n-f)\|_{\infty}\leq\supn{\lambda}\|f_n-f\|_{\infty}\to 0 \text{ as } n\to\infty.
 \end{align*}
 This finishes the proof.
\end{proof}
\section{MCMC algorithms fulfilling the assumptions}
\label{sec:examples}
In this section, we give examples of PDMPs fulfilling the assumptions of this work where we will focus on recently proposed MCMC schemes. All algorithms aim to provide samples from a probability distribution $\pi$ on $\Rd$ with  
\begin{align}
\label{eq:target_dist}
\pi(dq)=c^{-1}_{0}\exp{(-U(q))}dq,
\end{align}
where $U:\Rd\to\Reals$ is interpreted as potential and $c_0$ is the normalizing constant.\\
Instead of sampling directly from $\pi$, all these MCMC schemes
simulate Markov processes on an enlarged state space $\Rd\times\Rd$ with invariant distribution $\pi\otimes\mu$ where $\mu=\mathcal{N}(0,\mathbf{1}_d)$ (alternatively $\mu=\text{Unif}(\mathbb{S}^{d-1})$). For an element $z=(q,p)\in\Rdd$, we interpret $q\in\Rd$ as the position and $p$ as the corresponding velocity. 
\subsection{Randomized Hamiltonian Monte Carlo}
Firstly, we discuss the Randomized Hamiltonian Monte Carlo method (RHMC) \cite{RHMCintro}. While for the popular Hamiltonian Monte Carlo method the times between jumps are constant, here these times are $\text{Exp}(\lambda_{\text{ref}})$-distributed. Considered as a PDMP, this means that we have a constant intensity function $\lambda(z)=\lambda_{\text{ref}}$. The vector field $g$ is given by laws of physics, i.e. $g(q,p)=(p,-\nabla U(q))$ where $U:\Rd\to\Reals$ is interpreted as a potential and $\nabla U$ as force. We assume throughout that $U\geq 0$. The physical law of conservation of energy then imply the invariance of the Boltzmann-Gibbs distribution $\pi\otimes\mathcal{N}(0,\mathbf{1}_d)$ under these dynamics \cite{Actanumerica}. Refreshments only act on the velocities: $R^{\xi}(q,p)=(q,\alpha p+\sqrt{1-\alpha^{2}}\xi)$ where $\xi\sim\nd{0}{I_{d}}$ and $0\leq\alpha\leq 1$.
\begin{corollary}
	Assume that $0<\alpha\leq 1$, the force $\nabla U$ is Lipschitz-continuous and the potential fulfils
	\begin{align}
	\lim\limits_{\|z\|\to\infty}U(z)=\infty.
	\end{align}
	Then it holds that the RHMC process is Feller and $C_{c}^{\infty}(\Rdd)$ is a core of its generator.
\end{corollary}
\begin{proof}
If $\nabla U$ is Lipschitz then the vector field $g$ is Lipschitz, so Assumption \ref{Lipschitz} is true. Since $\|DR^{\xi}(q,p)\|= 1$ (in particular $\leq 1$), it satisfies  Assumption \ref{Jacobirefreshmentbound}. Finally, we show that condition (2) of Assumption \ref{Feller_assumptions} is true. The flows preserves energy given by the Hamiltonian $H$ \cite{Actanumerica}, i.e. for $z=(q,p)\in\Rdd$:
\begin{align*}
H(\varphi_t(z))=H(z),\quad \text{  where }
H(z)=U(q)+\frac{1}{2}\|p\|^2.
\end{align*}
Let $R>0$ be arbitrary. By assumption, $H^{-1}([0,L])$ is compact for every $L>0$. By the Heine-Borel theorem, there is an $L_0>0$ such that $B(0,R)=\{x\in\Rd | \|x\|< R\}\subset H^{-1}([0,L_0])=:K$. For all $z$ outside of $K$, it holds that $H(\varphi_t(z))=H(z)>L_0$ whereby it follows $\|\varphi_t(z)\|>R$, i.e.
\begin{align*}
	\inf\limits_{t\geq 0}\|\varphi_t(z)\|\geq R.
\end{align*}
For $\alpha>0$, it is clear that $R^\xi(z)\to\infty$ as $\|z\|\to\infty$. This shows that condition (2) of Assumption \ref{Feller_assumptions} holds. The statement follows by Assumption \ref{Feller} and \cref{Cinftycore}.
\end{proof}
In \cite{DELIG}, it was proven that RHMC is Feller under similiar assumptions. But instead of using resolvents and semigroup theory, we can give a direct proof using the results of this work.\\
The assumption that $\alpha>0$ is crucial for the proof. Intuitively, if $\alpha=0$, the process "forgets" its velocity after each jump. So even if our process $Z_t$ starts at $z=(q,p)$ and it holds that $\|p\|\to\infty$ (and thereby $\|z\|\to\infty$) we cannot infer that $\|Z_t\|$ goes to infinity as well.
\subsection{Isometric refreshment process - Zig-Zag, Pure Reflection Process}
Now, we consider PDMPs where velocities are constant in deterministic intervals, i.e. the flow is 
\begin{equation}
\label{BPSflow}
\varphi_{t}(q,p)=(q+tp,p).
\end{equation} 
In particular, the vector field is $g(q,p)=(p,0)$ which is clearly Lipschitz. So Assumption \ref{Lipschitz} holds trivially. Let $\lambda$ be an arbitrary continuous intensity function. Consider refreshments by $(q,p)\mapsto (q,O^{\xi,q}(p))$ where $O^{\xi,q}\in\mathcal{O}(d)$ is a random orthogonal matrix. We will call this an \emph{Isometric Refreshment process}. Examples include the Zig-Zag process, where $R^{\xi}(q,p)=(q,p_{1},...,p_{i(\xi)-1},-p_{i(\xi)},p_{i(\xi)+1}...,p_{d})$ with $i(\xi)$ is a random index \cite{bierkens2016zigzag}, the Pure Reflection process, where $R:(q,p)\mapsto(q,-p)$ \cite{Pure_reflection}, and as we will see also a version of the Bouncy Particle Sampler \cite{BPSwork}.
\begin{corollary}
	\label{cor:core_for_irp}
	An isometric refreshment process (e.g. Zig-Zag or a Pure Reflection process) is Feller and it holds that $C_{c}^{\infty}(\Rdd)$ is a core of the generator of the corresponding semigroup.
\end{corollary}
\begin{proof}
The only non-trivial thing to show is condition (2) of Assumption \ref{Feller_assumptions}. Fix $t>0$. If the process starts at $z=(q_0,p_0)$, 
we know that for some $p_1,\dots,p_l$ such that $\|p_l\|=\|p_0\|$ it holds
\begin{align*}
Z_t=(q_0+\sum\limits_{l=0}^{k}t_lp_l,p_k).
\end{align*}
In particular, $\|Z_t\|\geq\|p_k\|=\|p_0\|$ and $\|Z_t\|\geq\|q_0\|-\sum\limits_{l=0}^{k}t_l\|p_l\|=\|q_0\|-t\|p_0\|$. Therefore,
 		\begin{align*}
 			\inf\limits_{\substack{t_0+\dots+t_k=t\\t_i\geq 0\\
 			\xi_1,\dots,\xi_{k}\in \mathcal{S}}}\|\left [\varphi_{t_k}\circ R^{\xi_k}\circ\dots\circ R^{\xi_{1}}\circ\varphi_{t_0}\right ](z)\|\geq \max(\|p_0\|,\|q_0\|-t\|p_0\|).
 		\end{align*}
As $\|z\|=\|(q_0,p_0)\|\to\infty$, the right-hand side goes to infinity. Hence, condition (2) of Assumption \ref{Feller_assumptions} is fulfilled and we can apply Lemma \ref{Feller} and Theorem \ref{Cinftycore}.
\end{proof}
\subsection{Bouncy Particle Sampler}
The Bouncy Particle Sampler (BPS) was  introduced in \cite{BPSwork,BPSphysics}. Again, the state space is $\Rd\times\Rd$ and the vector field is $g(q,p)=(p,0)$. But the bouncy particle sampler admits two kinds of jump mechanism: bounces and velocity refreshments. If a bounce occurs at state $z=(q,p)$, it is done in the same way as a particle would change his velocity after a collision with the hyperplane  $\{\nabla U(q)\}^{\perp}$, i.e. the coordinates $q$ stay constant but the velocity changes by
\begin{equation}
\label{bouncedef}
p\mapsto p-2\frac{\eukl{\nabla U(q)}{p}}{\|\nabla U(q)\|^{2}}\nabla U(q)=:\mathcal{R}(q)p.
\end{equation}
The intensity for bounces are given by $\lambda(q,p)=\max(\eukl{\nabla U(q)}{p},0)=:\eukl{\nabla U(q)}{p}_{+}$.\\
Velocity refreshments are similiar to the ones of the RHMC: it is governed by a constant intensity $\lref>0$ and done by $(q,p)\to(q,\alpha p +\sqrt{1-\alpha^{2}}\xi)$ where $\xi\sim\nd{0}{I_{d}}$ with $0\leq\alpha\leq 1$ \cite{BPSwork}.
\begin{lemma}
	\label{BPSwithoutbouncesisFeller}
	For the Bouncy Particle Sampler with only bounces as refreshments, i.e. $\alpha=1$, the underlying PDMP is Feller and $C_{c}^{\infty}(\Rdd)$ is a core of its generator.
\end{lemma}
\begin{proof}
	By direct computation, one can see that $\mathcal{R}(q)$ is an orthogonal map and the BPS with $\alpha=1$ is an isometric refreshment process. The statement follows by Theorem \ref{cor:core_for_irp}.
\end{proof}
In practice, the Bouncy Particle Sampler is used with refreshments:
\begin{corollary}
	For autoregressive velocity refreshments, i.e. $\alpha>0$, the Bouncy Particle Sampler is Feller and $C_{c}^{\infty}(\Rdd)$ is a core of its generator.
\end{corollary}
\begin{proof}
	The proof that the process is Feller is similiar to the proof of Proposition  \ref{Feller} - each jump mechanism fulfils one of the two conditions in Assumption \ref{Feller_assumptions}.\\
	To show that $C_c^\infty(\Rdd)$ is a core, let, as before, $0<\alpha\leq 1$. Let  $\mathcal{L}'$ be the generator of the BPS with $\alpha=1$ (only bounces) and $\mathcal{L}$ the generator of the BPS as assumed in the corollary. Then by \cite[theorem 26.14]{DAVIS}
	\begin{equation}
	\mathcal{L}=\mathcal{L}'+B
	\end{equation}
	where $Bf:=\lref(\mathbb{E}[f(q,\alpha p +\sqrt{1-\alpha^{2}}\xi)]-f)$ with $\xi\sim\mathcal{N}(0,I_d)$. Since $B$ is a bounded operator, the core $C_{c}^{\infty}(\Rdd)$ of $\mathcal{L}'$ will also be a core of $\mathcal{L}$. So Lemma \ref{BPSwithoutbouncesisFeller} implies the statement. Alternatively, the proof for Theorem \ref{Cinftycore} can be performed in the same way having two summands for each jump mechanism.
\end{proof}

\section{Applications}
\label{sec:applications}
\subsection{Invariance of probability measures}

Let $\mu$ be a probability measure on $\Rd$, $\mathcal{L}$ the generator of a Markov process $Z_t$ on $\Rd$ and $D$ a core of $\mathcal{L}$. By \cite[Prop. 9.9.2]{Ethier}, the condition
\begin{align*}
\int\mathcal{L}fd\mu=0,\quad \text{for all } f\in D
\end{align*}
implies that $\mu$ is an invariant distribution of $Z_t$, i.e. if $Z_0\sim\mu$ then also $Z_t\sim\mu$ for all $t>0$.\\
By \Cref{Cinftycore}, we can provide very simple proofs for the invariance of the target distribution under the afore-mentioned MCMC schemes since we can choose $D=C_c^\infty(\Rdd)$ and use the  regularity of functions in $C_c^\infty(\Rdd)$.\\
As an illustration, we outline this for RHMC and invariant distribution $\mu=\pi\otimes\mathcal{N}(0,\mathbf{1}_d)$ with $\pi$ given as in \cref{eq:target_dist}. We compute using \cref{eq:generator}
\begin{align*}
\int \mathcal{L}f \mu(dz)=\int\eukl{\nabla f}{(p,\nabla U(q)}\mu (dz)+\lref\int \mathbb{E}[f(q,\alpha p +\sqrt{1-\alpha^2}\xi)]-f(z)\mu(dz).
\end{align*}
Due to the regularity of $f$, one can use integration by parts to show that the first term vanishes. 
The second one vanishes since for independent $p,\xi\sim\mathcal{N}(0,\mathbf{1}_d)$ also $\alpha p+\sqrt{1-\alpha^2}\xi\sim\mathcal{N}(0,\mathbf{1}_d)$. In \cite{RHMCintro} where RHMC was introduced, the authors performed a similiar computation - but here, we can stop at this point since we know $C_c^\infty(\Rdd)$ is a core.
\subsection{Martingale problems}
As another application of the previous results, we present an equivalent characterization of a PDMP process by martingale problems \cite[chapter 4]{Ethier}. For example, such a characterization is used in the analysis of scaling limits of PDMPs \cite{DELIG}.  Define $\mathcal{A}:=\sset{(f,\mathcal{L}f)}{f\in C_{c}^{\infty}(\Rd)}$. 
\begin{corollary}
	\label{uniquenessmartingaleproblem}
	The PDMP process $Z_t$ fulfilling Assumption \ref{Lipschitz},  \labelcref{Jacobirefreshmentbound}, and \labelcref{Feller_assumptions} with initial distribution $Z_0\sim v$ is the unique solution of the martingale problem for $(\mathcal{A},v)$, i.e. if $Y_{t}$ is any progressively measurable process on $\Rd$ such that $Y_{0}\sim v$ and for all $f\in C_{c}^{\infty}(\Rd)$
	\begin{align}
	M_{t}^{Y,f}:=f(Y_{t})-\int\limits_{0}^{t}\mathcal{L}f(Y_{s})ds	
	\end{align}
	is a martingale, then $Y$ is a Markov process with the same transition semigroup as $Z_t$.
\end{corollary}
\begin{proof}
	Since $P_t$ is Feller, it is well-known that all solutions of the martingale problem for $(\mathcal{A}',v)$ with $\mathcal{A}'=\sset{(f,\mathcal{L}f)}{f\in \text{dom} (\mathcal{L})}$ have the same transition semigroup as $Z_t$ \cite[theorem 1.2.6, theorem 4.4.1]{Ethier}.\\
	Now, let $Y_{t}$ be a solution for $(\mathcal{A},v)$ and $f\in \text{dom} (\mathcal{L})$ be arbitrary. Since $C_{c}^{\infty}(\Rd)$ is a core of $\text{dom} (\mathcal{L})$, we can find $f_{n}\in C_{c}^{\infty}(\Rd)$ such that $\supn{\mathcal{L}f_{n}-\mathcal{L}f},\supn{f_{n}-f}\to 0$. It can be easily seen that then for every $t\geq 0$: $M_{t}^{Y,f_{n}}\to M_{t}^{Y,f}$ in $L^1$ (convergence in mean). Since $M_{t}^{Y,f_{n}}$ is a martingale, this implies that $M_{t}^{Y,f}$ is a martingale. So every solution of the martingale problem for $(\mathcal{A},v)$ is a solution for $(\mathcal{A}',v)$. Hence, the uniqueness of $(\mathcal{A}',v)$ implies the statement.
\end{proof}






\ACKNO{I want to express my gratitude to Andreas Eberle for his helpful advice and extensive guidance across the whole course of this project. I also want to thank Nawaf Bou-Rabee for his advice on improving the manuscript. Finally, I would like to thank Benedikt Geiger for his detailed suggestions for improvements of this work.}


\end{document}